\documentclass[11pt]{amsart}

\usepackage{amsmath,xy,leftidx,amsthm,sfheaders}
\usepackage{xspace}
\usepackage[psamsfonts]{amssymb}
\usepackage[latin1]{inputenc}
\usepackage{graphicx,color}

\usepackage{hyperref}

\theoremstyle{remark}

\newtheorem*{remark}{\bf Remark}

\theoremstyle{plain}

\newtheorem{proposition}{\bf Proposition}[section]
\newtheorem{definition}[proposition]{\bf Definition}
\newtheorem{theorem}[proposition]{\bf Theorem}
\newtheorem{theoremintro}{\bf Theorem}

\newtheorem{lemma}[proposition]{\bf Lemma}
\newtheorem{corollary}[proposition]{\bf Corollary}

\def\C{{\mathbb C}}
\def\R{{\mathbb R}}
\def\Z{{\mathbb Z}}

\def\B{\mathbb{B}}
\def\D{\mathbb{D}}

\def\J{\mathcal{J}}
\def\K{\mathcal{K}}
\def\p{\mathbb{P}}

\def\supp{\textup{supp}}

\def\pe{\textup{ := }}
\def\rat{\textup{Rat}}
\def\Per{\textup{Per}}
\def\bif{\textup{bif}}

\def\Mand{\textbf{\textup{M}}}
\def\J{\mathcal{J}}

\def\and{{\quad\text{and}\quad}}

\title{Higher bifurcation currents, neutral cycles and the Mandelbrot set}

\author{Thomas Gauthier}

\begin{document}

\maketitle

\begin{abstract}
We prove that given any $\theta_1,\ldots,\theta_{2d-2}\in \R\setminus\Z$, the support of the bifurcation measure of the moduli space of degree $d$ rational maps coincides with the closure of classes of maps having $2d-2$ neutral cycles of respective multipliers $e^{2i\pi\theta_1},\ldots,e^{2i\pi\theta_{2d-2}}$. To this end, we generalize a famous result of McMullen, proving that homeomorphic copies of $(\partial \Mand)^{k}$ are dense in the support of the $k^{th}$-bifurcation current $T^k_\bif$ in general families of rational maps, where $\Mand$ is the Mandelbrot set. As a consequence, we also get sharp dimension estimates for the supports of the bifurcation currents in any family.
\end{abstract}

\section{Introduction.}

Given $d\geq 2$, the \emph{bifurcation locus} of any holomorphic family $(f_\lambda)_{\lambda\in\Lambda}$ of degree $d$ rational maps (or of the moduli space $\mathcal{M}_d$ of degree $d$ rational maps) is the closure of the set of discontinuity of the map $\lambda\mapsto \J_\lambda$, where $\J_\lambda$ is the Julia set of $f_\lambda$. DeMarco \cite{DeMarco1} has shown that the bifurcation locus of $\Lambda$ is the support of a closed positive $(1,1)$-current $T_\bif$ which is called the \emph{bifurcation current} of the family $(f_\lambda)_{\lambda\in\Lambda}$. When $(f_\lambda)_{\lambda\in\Lambda}$ comes with $2d-2$ marked critical points $c_1,\ldots,c_{2d-2}$, the current $T_\bif$ coincides with $\sum_iT_i$, where $T_i$ is the bifurcation current of the critical point $c_i$ (see \cite{DeMarco2}). Bassanelli and Berteloot \cite{BB1} initiated the study of the self-intersections $T_\bif^k$, $1\leq k\leq \min(2d-2,\dim\Lambda)$, of the bifurcation current. These currents give a natural stratification of the bifurcation locus by loci of stronger bifurcations and are well-adapted to the study of the complex geometric properties of the bifurcation locus. We refer the reader to the survey \cite{dsurvey} or the lecture notes \cite{bsurvey} for a report on recent results involving bifurcation currents and further references. We also refer to section 2 for precise definitions.

~ 

\par Several different descriptions of the currents $T_\bif^k$ have been provided by various authors. Let us mention some known results. The set $\Per_n(w)$ of parameters $\lambda\in\Lambda$ for which $f_\lambda$ has a cycle of multiplier $w\in\C$ and exact period $n$ is a complex hypersurface of $\Lambda$. Bassanelli and Berteloot \cite{BB3} proved that the $k^{th}$ bifurcation current $T_\bif^k$ is actually the limit of integration currents of the form
$$\frac{d^{-(s_1(n)+\cdots+s_k(n))}}{(2\pi)^m}\int_{[0,2\pi]^k}\bigwedge_{j=1}^k[\Per_{s_j(n)}(re^{i\theta_j})]d\theta_1\cdots\theta_k~,$$
for any $r>0$ and a suitable choice of increasing functions $s_j:\mathbb{N}\longrightarrow\mathbb{N}$. In the family of all degree $d$ polynomials, they give in \cite{BB2} a much stronger result when $k=1$: they prove that the hypersurfaces $d^{-n}[\Per_n(re^{i\theta})]$ converge to $T_\bif$ for \emph{fixed} $r\leq1$ and $\theta\in\R$. Regarding Bassanelli and Berteloot's work, one can expect the current $T_\bif^k$ to be the limit of currents of the form $d^{-(s_1(n)+\cdots+s_k(n))}[\Per_{s_1(n)}(re^{i\theta_1})]\wedge\cdots\wedge[\Per_{s_k(n)}(re^{i\theta_k})]$ for fixed $\theta_i\in\R$ and $r$. Recently, Favre and the author \cite{distriPCF} gave an affirmative answer to this problem in the case when $r<1$ and $k=d-1$ in the family of all degree $d$ polynomials, using a Theorem of Yuan \cite{yuan} concerning the equidistribution of small points.  This problem remains wide open when $r>1$.
\par In this paper we focus on a different question of topological nature, namely,  whether parameters possessing $k$ distinct
neutral cycles of given multipliers are dense in the support of $T_\bif^k$. In the whole paper, we consider connected holomorphic families of rational maps. Our first result can be formulated as follows.

\begin{theoremintro}
Let $T_\bif$ be the bifurcation current of the moduli space $\mathcal{M}_d$ of degree $d$ rational maps. For any $1\leq k\leq 2d-2$ and any $\Theta_k=(\theta_1,\ldots,\theta_k)\in(\R\setminus\Z)^k$,
\begin{center}
$\supp\left(T_\bif^k\right)=\overline{\mathcal{Z}_k(\Theta_k)}=\overline{\textup{Prerep}(k)}$,
\end{center}
where $\textup{Prerep}(k)\pe\{[f]\in\mathcal{M}_d \,;\,f\text{ has }k\text{ critical points preperiodic to repelling cycles}\}$ and $\mathcal{Z}_k(\Theta_k)\pe\{[f]\in\mathcal{M}_d \,;\,f\text{ has }k\text{ distinct cycles of resp. multipliers }e^{2i\pi\theta_1},\ldots,e^{2i\pi\theta_k}\}$.
\label{maintheorem}
\end{theoremintro}

\par Let us mention that the equality $\supp(T_\bif^k)=\overline{\textup{Prerep}(k)}$ is known (see \cite{buffepstein,Article2,favredujardin} for the case when $k$ is maximal). Dujardin \cite[Corollary 5.3]{higher} proved it in the general case, using a transversality Theorem for laminar currents.

~

\par Let us now describe how we prove Theorem \ref{maintheorem}. The main point is to generalize McMullen's universality of the Mandelbrot set: McMullen \cite{McMullen3} proved that in any one-dimensional family of rational maps, the bifurcation locus contains quasiconformal copies of the Mandelbrot set $\Mand$. We prove here that under some mild assumptions, the loci of stronger bifurcations contain also copies of products of the Mandelbrot set with itself. Relying on \cite{McMullen3} and \cite{Article1}, we prove the following.

\begin{theoremintro}
Let $(f_\lambda)_{\lambda\in\D^m}$ be a holomorphic family of degree $d$ rational maps with simple marked critical points $c_1,\ldots,c_k$ with $k\leq m$. Assume that $c_1,\ldots,c_k$ are transversely preperiodic to repelling cycles of $f_0$. Then, for any $\epsilon>0$, there exists a continuous embedding $\Phi:\Mand^k\times\D^{m-k}\hookrightarrow\D^m$ and integers $n_1,\ldots,n_k\geq 1$ such that
\begin{enumerate}
\item for any $(\zeta_1,\ldots\zeta_k,t)\in\Mand^k\times\D^{m-k}$, if $\lambda=\Phi(\zeta_1,\ldots,\zeta_k,t)$, there exists $k$ disjoint compact sets $\K_1,\ldots,\K_k\subset\p^1$ such that $f_\lambda^{n_i}:\K_i\rightarrow\K_i$ is hybrid conjugate to $z^2+\zeta_i$.
\item the set $\Phi\big((\partial\Mand)^k\times\D^{m-k}\big)$ is contained in $\supp(T_1\wedge\cdots\wedge T_k)$ and $$\dim_H\Phi\big((\partial\Mand)^k\times\D^{m-k}\big)\geq 2m-\epsilon.$$
\end{enumerate}
\label{mainthmMcM}
\end{theoremintro}

 This generalization of McMullen's Theorem is done in section $3$. To prove Theorem \ref{mainthmMcM}, we use McMullen's universality for each critical point separately to produce $k$ ``tubes'' of Mandelbrot set homeomorphic to $\Mand\times\D^{m-1}$ and which are tranverse to each other. We then construct $\Phi$ as a map from $\Mand^k\times\D^{m-k}$ to the intersection of those tubes. Let us stress out that the dimension estimate uses Shishikura's famous result \cite{shishikura2} concerning the Hausdorff dimension of the Mandelbrot set and H\"older regularity properties of $\Phi$ (see Theorem \ref{propmuniv}). Using \cite[Theorem 6.2]{Article1}, we then prove that the copy of $(\partial \Mand)^k\times\D^{m-k}$ given by $\Phi$ actually lies in the support of $T_1\wedge\cdots\wedge T_k$ (see Proposition \ref{cormuniv}).
\par Let us also mention that Inou and Kiwi \cite{kiwiinou} and Inou \cite{inou} have already obtained strengthened versions of McMullen's unversality of the Mandelbrot set in a different setting and given an explicit condition for the related embedding to be not continuous. On the other hand, Buff and Henriksen \cite{buffhenriksen} proved that some parameter spaces contain quasiconformal copies of Julia sets.

~

\par In \cite{Article1}, the author obtained sharp dimension estimates for the strong bifurcation loci of the space $\rat_d$ of all degree $d$ rational maps. Using Theorem \ref{mainthmMcM}, we actually get sharp  estimates for the Hausdorff dimension of the strong bifurcation loci of a general family. This is the subject of our third result.

\begin{theoremintro}
Let $(f_\lambda)_{\lambda\in\Lambda}$ be a holomorphic family of degree $d$ rational maps. Assume that there exists $\lambda_0$ such that $f_{\lambda_0}$ has simple critical points and let $1\leq k\leq 2d-2$ be such that $T_\bif^k\neq0$. Then $\supp(T_\bif^k)\setminus\supp(T_\bif^{k+1})\neq\emptyset$ and for any open set $\Omega\subset\Lambda$ such that $\Omega\cap\supp(T_\bif^k)\setminus\supp(T_\bif^{k+1})\neq\emptyset$, we have
\begin{center}
$\dim_H\big(\Omega\cap\supp(T_\bif^k)\setminus\supp(T_\bif^{k+1})\big)=2\dim_\C\Lambda$.
\end{center}
\label{tmdim}
\end{theoremintro}

Let us also remark that our results strongly rely on \cite[Theorem 6.2]{Article1} and that Theorems \ref{maintheorem} and \ref{tmdim} also rely on \cite[Theorem 0.1]{higher}. The main difference with the proof of Theorem 1.1 of \cite{Article1} is the transfer phenomenom which is performed. Instead of transferring directly ``big'' sets from the dynamical space to the parameter space, we transfer a complete ``simplified'' parameter space into our actual parameter space.

~

\par Section 4 is devoted to explaining how to apply results from the previous sections to the particular case of the space $\rat_d^{cm}$ of all critically marked degree $d$ rational maps in order to obtain Theorem \ref{maintheorem}. We also prove a similar result, using a simpler argument, in the case of the moduli space $\mathcal{P}_d^{cm}$ of critically marked degree $d$ polynomials.

\paragraph*{Aknowledgements}
The author would like to thank Fran\c{c}ois Berteloot, Xavier Buff, Arnaud Ch\'eritat, Romain Dujardin, Charles Favre and Carsten Petersen without whose precious advice and knowledge this paper would never have appeared. The author would also like to thank the IMS and Stony Brook University which he was visiting during the autumn 2012 and where he finished the elaboration of the present work and the referee for his useful remarks.

\section{Preliminaries.}
\par Let us begin with introducing some tools and recalling known results we will need.
\subsection{The hypersurfaces $\Per_n(w)$.}\label{sectionPern}
\par To understand the geometry of the bifurcation locus of a holomorphic family of rational maps, one can investigate the geometry of the set of rational maps having a cycle of given multiplier and period. The following result describes the set of such parameters (see \cite[Chapter 4]{Silverman}):

\begin{theorem}[Silverman]
Let $(f_\lambda)_{\lambda\in \Lambda}$ be a holomorphic family of degree $d$ rational maps. Then for any $n\in\mathbb{N}^*$ there exists a holomorphic function $p_n:\Lambda\times\C\longrightarrow\C$ such that :
\begin{enumerate}
	\item For any $w\in \C\setminus\{1\}$, $p_n(\lambda,w)=0$ if and only if $f_\lambda$ has a cycle of exact period $n$ and of multiplier $w$,
	\item $p_n(\lambda,1)=0$ if and only if $f_\lambda$ has a cycle of period $n$ and multiplier $1$ or $f_\lambda$ has a cycle of period $m$ and multiplier a $r$-th root of unity with $n=mr$,
	\item for any $\lambda\in X$, the function $p_n(\lambda,\cdot)$ is a polynomial of degree $N_d(n)\sim\frac{1}{n}d^n$.
\end{enumerate}
Moreover, if $\Lambda$ is a quasi-projective variety, the functions $p_n$ are polynomials in $(\lambda,w)$.\label{tmdefPern}
\end{theorem}

\par For $n\geq1$ and $w\in\C$ we set $\Per_n(w)\pe\{\lambda\in \Lambda \ | \ p_n(\lambda,w)=0\}$. We will say that a neutral periodic point of $f_{\lambda_0}$ is \emph{persistent} in $\Lambda$ if it can be perturbed as a neutral periodic point of $f_\lambda$ for any $\lambda$ in a neighborhood of $\lambda_0$ in $\Lambda$, i.e. that $\Per_n(e^{i\theta})=\Lambda$ for some $n,\theta$.

\subsection{Bifurcation current of a critical point.}\label{sectioncrit}
\par Let $(f_\lambda)_{\lambda\in\Lambda}$ be a holomorphic family of degree $d$ rational maps. We say that $c$ is a \emph{marked critical point} if $c:\Lambda\longrightarrow\p^1$ is a holomorphic map satisfying $f_\lambda'(c(\lambda))=0$ for every $\lambda\in \Lambda$. If $\deg(f_\lambda,c(\lambda))=2$ for any $\lambda\in \Lambda$, we will say the the marked critical point $c$ is \emph{simple}.

\begin{definition}
\par We say that a marked critical point $c$ is \emph{passive} at $\lambda_0$ in $\Lambda$ if $(f_\lambda^n(c(\lambda)))_{n\geq0}$ is a normal family in a neighborhood of $\lambda_0$. Otherwise we say that $c$ is \emph{active} at $\lambda_0$ in $\Lambda$.\label{actif}
\end{definition}

\par Let $\omega$ be the Fubini-Study form on $\p^1$ and denote by $c_n(\lambda)\pe f_\lambda^{\circ n}(c(\lambda))$. Dujardin and Favre prove in \cite[Section 3.1]{favredujardin} that the sequence $d^{-n}c_n^*\omega$ converges to a positive closed $(1,1)$-current $T_c$ with local continuous potential, which support coincides with the activity locus of the marked critical point $c$.

\begin{definition}
$T_c$ is called the \emph{bifurcation current} of the marked critical point $c$.
\end{definition}

\par As $T_c$ has local continuous potential, the self-intersections of $T_c$ are well-defined in the sense of Bedford and Taylor (see \cite{bedfordtaylor}). The bifurcation current of a critical point has self-intersection zero (see \cite[Proposition 6.9]{favredujardin} for polynomial families and \cite[Theorem 6.1]{Article1} for the general case).

\begin{lemma}[Dujardin-Favre, Gauthier]
Let $(f_\lambda)_{\lambda\in \Lambda}$ be a holomorphic family of degree $d$ rational maps with a marked critical point $c$, then $T_c\wedge T_c=0$.
\label{lmTc}
\end{lemma}

\par Assume that $(f_\lambda)_{\lambda\in\Lambda}$ is with $2d-2$ marked critical points $c_1,\ldots,c_k$ and $\dim\Lambda\geq k$ and let us set $H_i(k_i,p_i)\pe\big\{\lambda\in\Lambda \ | \ f_{\lambda}^{\circ(k_i+p_i)}(c_i(\lambda))=f_{\lambda}^{\circ p_i}(c_i(\lambda))$ and $f_{\lambda}^{\circ p_i}(c_i(\lambda))$ is repelling$\big\}$, for $1\leq i\leq k$.

\begin{definition}
If $\lambda_0\in\bigcap_{1\leq i\leq k}H_i(k_i,p_i)$, we say that $c_1,\ldots,c_k$ \emph{fall transversely onto repelling cycles} at $\lambda_0$ if the hypersurfaces $H_i$ are smooth at $\lambda_0$ and intersect transversely at $\lambda_0$. If they only intersect properly, we say that $c_1,\ldots,c_k$ \emph{fall properly onto repelling cycles} at $\lambda_0$.
\label{deftransvers}
\end{definition}

Dujardin \cite{higher} proved the following which we will use for proving Theorems \ref{maintheorem} and \ref{tmdim}.

\begin{theorem}[Dujardin]
Let $(f_\lambda)_{\lambda\in \Lambda}$ be a holomorphic family of degree $d$ rational maps with $2d-2$ marked critical points $c_1,\ldots,c_k$ and let $T_1,\ldots,T_k$ be their respective bifurcation currents. Then
\begin{center}
$\displaystyle\supp(T_1\wedge\cdots\wedge T_k)=\overline{\{\lambda\in\Lambda \ | \ c_1,\ldots,c_k\text{ fall transversely onto repelling cycles}\}}$.
\end{center}
\label{tmduj}
\end{theorem}

\subsection{The bifurcation currents of a holomorphic family.}
\par Every rational map $f$ of degree $d\geq2$ on the Riemann sphere admits a unique maximal entropy measure $\mu_f$. The Lyapunov exponent of $f$ with respect to $\mu_f$ is defined by
$$L(f)=\int_{\p^1}\log|f'|\mu_f~.$$
It turns out that, for any holomorphic family $(f_\lambda)_{\lambda\in\Lambda}$ of degree $d$ rational maps, the function $L:\Lambda\longrightarrow L(f_\lambda)$ is \emph{p.s.h} and continuous on $\Lambda$ (see \cite{DeMarco1}).

\begin{definition}
The \emph{bifurcation current} of the family $(f_\lambda)_{\lambda\in\Lambda}$ is the closed, positive $(1,1)$-current on $\Lambda$ defined by $T_\bif\pe dd^cL$.
\end{definition}

\par The support of $T_{\bif}$ coincides with the bifurcation locus of the family $(f_\lambda)_{\lambda\in \Lambda}$ in the sense of Ma\~n\'e-Sad-Sullivan, i.e. the closure of the set of discontinuity of the map $\lambda\in\Lambda\longmapsto \mathcal{J}_\lambda$. This actually follows from a formula by DeMarco (see \cite[Theorem 1.1]{DeMarco2} or \cite[Theorem 5.2]{BB1}), which, for families with $2d-2$ marked critical points $c_1,\ldots,c_{2d-2}$, may be stated as follows:
\begin{eqnarray*}
T_\bif=\sum_{i=1}^{2d-2}T_i.
\end{eqnarray*}

\begin{definition}
Let $1\leq k\leq \min(2d-2,\dim \Lambda)$. The $k^\textup{th}$-\emph{bifurcation current} of the family $(f_\lambda)_{\lambda\in \Lambda}$ is the closed positive $(k,k)$-current defined by $T_\bif^k\pe(dd^cL)^k$.
\end{definition}

\par Lemma \ref{lmTc} directly gives for $1\leq k\leq 2d-2$:
\begin{eqnarray}
T_\bif^k=k!\sum_{i_1<\cdots<i_k}T_{i_1}\wedge\cdots\wedge T_{i_k}.
\label{Demarco2}
\end{eqnarray}
The locus $\supp(T_\bif^k)$ can thus be interpreted as the set of parameters for which at least $k$ critical points are active in an ``independent'' manner.\nocite{Article2}

\subsection{Quadratic-like maps.}
\par Let $U,V\subset\C$ be topological discs with $U\Subset V$. We say that $f:U\rightarrow V$ is a \emph{quadratic-like map} if it is a degree $2$ branched cover. The \emph{filled-in Julia set} $\K(f)$ of $f$ is the set
\begin{center}
$\displaystyle\K(f)\pe\bigcap_{n\geq1}f^{-\circ n}(V)$
\end{center}
of points $z\in U$ such that $f^{\circ n}(z)\in V$ for any $n\geq1$. We say that the map $f$ is \emph{hybrid conjugate} to a quadratic polynomial $p_\zeta(z)\pe z^2+\zeta$ if there exists a quasi-conformal map $\varphi$ from a neighborhood of $\K_\zeta\pe\K(p_\zeta)$ to a neighborhood of $\K(f)$ which satisfies $\varphi\circ p_\zeta=f\circ\varphi$ and $\overline{\partial}\varphi=0$ on $\K_\zeta$.
\par Douady and Hubbard proved that for any holomorphic family of quadratic-like maps, the Mandelbrot set plays the role of a good model. Let us summarize here the properties of quadratic-like maps established by Douady and Hubbard (see \cite[Proposition 13 and Chapter IV]{DH}).

\begin{theorem}[Douady-Hubbard]
Let $(f_\lambda)_{\lambda\in \Lambda}$ be a holomorphic of quadratic-like maps parametrized by a complex manifold $\Lambda$. Let $\Mand_\Lambda\pe\{\lambda\in \Lambda \ | \ \K(f_\lambda)$ is connected$\}$. There exists a continuous map $\chi:\Mand_\Lambda\longrightarrow\Mand$ such that:
\begin{enumerate}
\item $\chi$ is holomorphic from $\mathring{\Mand}_\Lambda$ to $\mathring{\Mand}$,
\item for any $\lambda\in\Mand_\Lambda$, if $\chi(\lambda)=\zeta$, the map $f_\lambda$ is hybrid conjugate to $z^2+\zeta$ on $\K(f_\lambda)$,
\item for all $\zeta\in\Mand$, the set $\chi^{-1}\{\zeta\}$ is an analytic hypersurface,
\item if $\dim\Lambda=1$ and $\lambda_0\in\Mand_\Lambda$, then there exists a neighborhood $V\subset \Lambda$ of $\lambda_0$ such that either $\chi$ is constant along $V$ or $\chi(V)$ contains a neighborhood of $\chi(\lambda_0)$ in $\Mand$.
\end{enumerate}
\label{tmDH}
\end{theorem}

\par The map $\chi$ defined in Theorem \ref{tmDH} is called the \emph{straightening map} of the family $(f_\lambda)_{\lambda\in\Lambda}$. Denote by \rotatebox[origin=c]{270}{$\heartsuit$} the main cardioid of the Mandelbrot set $\Mand$. Combined with the fact that the multiplier of the non-repelling fixed point parametrizes the closure of \rotatebox[origin=c]{270}{$\heartsuit$}, Theorem \ref{tmDH} gives (see also \cite[Section 3.2]{BB2} for a proof based on potential theoretic arguments):

\begin{corollary}[Bassanelli-Berteloot, Douady-Hubbard]
For any $\theta\in\R$, the set of $\zeta\in\C$ for which $p_\zeta$ has a cycle of multiplier $e^{2i\pi\theta}$ is dense in $\partial\Mand$.
\label{density}
\end{corollary}

\par Let $g_\zeta(z)\pe p_\zeta(z)+h(\zeta,z)$ be a holomorphic family of maps defined for $(\zeta,z)\in\D(0,R)\times\D(0,R)$, where $R>10$ and $g_\zeta'(0)=0$. Denote by $M_g$ the set of $\zeta\in\D(0,R)$ such that the orbit $(g^{\circ n}_\zeta(0))_n$ remains in $\D(0,R)$ for any $n>0$. In what follows, we will use the following lemma which is due to McMullen (see \cite[Lemma 4.2]{McMullen3}):

\begin{lemma}[McMullen]
There exists $\delta>0$ such that if $\sup_{(\zeta,z)}|h(\zeta,z)|=\epsilon<\delta$ then there exists a homeomorphism $\varphi:\Mand\longrightarrow M_g$ such that:
\begin{enumerate}
\item $g_{\phi(\zeta)}$ is hybrid conjugate to $p_\zeta$ for any $\zeta\in\Mand$,
\item $|\varphi(\zeta)-\zeta|<O(\epsilon)$,
\item $\varphi$ extends to a $(1+\epsilon/\delta)$-quasiconformal homeomorphism of $\C$.
\end{enumerate}
\label{lmMcM}
\end{lemma}

\section{The Mandelbrot set is universal, revisited.}

\par Let $(f_\lambda)_{\lambda\in\Lambda}$ be a holomorphic family of degree $d$ rational maps. In the present section, we want to prove that, under some reasonable condition on the family, the parameter space $\Lambda$ contains homeomorphically embedded copies of $\Mand^k\times\D^{\dim\Lambda-k}$, which generalizes the work \emph{The Mandelbrot set is universal} \cite{McMullen3} of McMullen. The main result of this section is the following.

\begin{theorem}
Let $(f_\lambda)_{\lambda\in\D^m}$ be a holomorphic family of degree $d$ rational maps with marked simple critical points $c_1,\ldots,c_k$ with $k\leq m$. Assume that $c_1,\ldots,c_k$ fall transversely onto repelling cycles at $0$. Then, for any $\epsilon>0$, there exists a homeomorphic embedding $\Phi:\Mand^k\times\D^{m-k}\longrightarrow\D^m$ and a continuous family $\{\varphi_{\zeta,t,i}:\p^1\longrightarrow\p^1\}_{(\zeta,t)\in\Mand^k\times\D^{m-k},1\leq i\leq k}$ of $(1+O(\epsilon))$-quasi-conformal homeomorphisms satisfying the following properties:
\begin{enumerate}
\item $\Phi(\zeta,\cdot):\D^{m-k}\longrightarrow\D^m$ is holomorphic for any $\zeta\in\Mand^k$,
\item $\Phi$ is holomorphic on $(\mathring{\Mand})^k\times\D^{m-k}$,
\item $\dim_H\big(\Phi((\partial\Mand)^k\times\D^{m-k})\big)\geq 2m-O(\epsilon)$,
\item for any $1\leq i\leq k$, there exists $n_i\geq1$ such that  $\varphi_{\zeta,t,i}\circ p_{\zeta_i}=f_{\Phi(\zeta,t)}^{\circ n_i}\circ \varphi_{\zeta,t,i}$ on $\K_{\zeta_i}$ and the conjugacy is hybrid.
\end{enumerate}
\label{propmuniv}
\end{theorem}	

It is the combination of this result with \cite[Theorem 6.2]{Article1} which will actually give Theorem \ref{mainthmMcM} (see Section \ref{sec:tm2}).

\subsection{Technical lemmas.}
To give Hausdorff dimension of estimates, we will need the two following lemmas. The first one is due to McMullen (see \cite[Lemma 5.1]{McMullen3}) and a proof of the second one is provided.

\begin{lemma}
Let $Y$ be a metric space and $X\subset Y\times[0,1]^k$. Denote by $X_t$ the slice $X_t\pe\{y\in Y \ | \ (y,t)\in X\}$. If $X_t\neq\emptyset$ for almost every $t\in[0,1]^k$, then
\begin{center}
$\dim_H\big(X\big)\geq k+\dim_H\big(X_t\big)$, for almost every $t$.
\end{center}
\label{lmdimHMcMullen}
\end{lemma}

\par Let us recall that a map $h:(X,d)\longrightarrow(Y,d')$ is $\alpha$-\emph{bi-H\"older} with constant $C>0$ if $0<\alpha\leq1$ and
\begin{center}
$C^{-1}d'(x,x')^{1/\alpha}\leq d(f(x),f(x'))\leq Cd(x,x')^\alpha$, for any $x,x'\in X$.
\end{center}

\begin{lemma}
Let $E_1,\ldots,E_k\subset\D$ and $f:E_1\times\cdots\times E_k\longrightarrow\C^k$ be a map. Assume that there exists $C>0$ and $0<\alpha\leq1$ such that for any $1\leq j\leq k$ and any $x_i\in E_i$ with $i\neq j$, for all $x,x'\in X_j$,
\begin{center}
$x\longmapsto f(x_1,\ldots,x_{j-1},x,x_{j+1},\ldots,x_k)$
\end{center}
is $\alpha$-bi-H\"older with constant $C$ and
\begin{center}
$f(\{x_1,\ldots,x_{j-1}\}\times E_j\times\{x_{j+1},\ldots,x_k\})\subset\{a_1,\ldots,a_{j-1}\}\times\C\times\{a_{j+1},\ldots,a_k\}$
\end{center}
for some $a_i\in\C$, $i\neq j$ only depending on $f$ and the $x_i$, $i\neq j$. Then $f$ is $\alpha$-bi-H\"older with constant $C\cdot\max\{k,k^{1/2\alpha}\}$. In particular,
$$\dim_H(f(E_1\times\cdots\times E_k))\geq\alpha\sum_{j=1}^k\dim_H(E_j)~.$$
\label{lmdimH}
\end{lemma}
\begin{proof}
Up to taking $C'\geq C$, we can assume that $C\geq1$. Let $E\pe E_1\times \cdots \times E_k$. Let $x,x'\in E$, then by assumption,
\begin{eqnarray*}
\|f(x)-f(x')\| & \leq & \sum_{j=1}^k\|f(x_1',\ldots,x_j',x_{j+1},\ldots,x_k)-f(x_1',\ldots,x_{j-1}',x_j,\ldots,x_k)\|\\
& \leq & C \sum_{j=1}^k|x_j-x_j'|^\alpha\leq C.k\|x-x'\|^\alpha.
\end{eqnarray*}
Again, by hypothesis, we have
\begin{center}
$\displaystyle\|x-x'\|\leq\sqrt{k}\max_{1\leq j\leq k}|x_j-x_j'|\leq C^\alpha\sqrt{k}\max_{1\leq j\leq k}\|f(x)-f(x_1,\ldots,x_{j-1},x',x_{j+1},\ldots,x_k)\|^\alpha$.
\end{center}
By assumption, $\|f(x)-f(x_1,\ldots,x_{j-1},x',x_{j+1},\ldots,x_k)\|=|(f(x))_j-(f(x'))_j|$ and thus
$$\displaystyle\|x-x'\|\leq C^\alpha\sqrt{k}\max_{1\leq j\leq k}|(f(x))_j-(f(x'))_j|^\alpha\leq C^\alpha \sqrt{k}
\|f(x)-f(x')\|^\alpha~.$$
The Hausdorff dimension estimate is classical (see e.g. \cite{falconer}).
\end{proof}

\subsection{Embeddings of $k$-fold products of $\Mand$: proof of Theorem \ref{propmuniv}}

By assumption, for any $1\leq i\leq k$, there exists integers $p_i,k_i\geq1$ such that
\begin{center}
$f_0^{\circ(k_i+p_i)}(c_i(0))=f_0^{\circ p_i}(c_i(0))$.
\end{center}
Denote by $a_i\pe f_0^{\circ k_i}(c_i(0))$. As $a_i$ is a repelling cycle of $f_0$, by the implicit function Theorem, up to reducing $\D^m$, we may assume that $a_i$ can be followed holomorphically on the whole $\D^m$ as a $p_i$-repelling cycle $a_i(\lambda)$ of $f_\lambda$. Let us now set:
\begin{eqnarray*}
\chi:\D^m & \longrightarrow & \C^k\\
\lambda & \longmapsto & \big(f_\lambda^{\circ p_1}(c_1(\lambda))-a_1(\lambda),\ldots,f_\lambda^{\circ p_k}(c_k(\lambda))-a_k(\lambda)\big).
\end{eqnarray*}
By assumption, up to reducing $\D^m$, the map $\chi$ is a submersion onto its image $\Omega$. The map $\chi$ allows us to defined a system of coordinates $(x_1,\ldots,x_m)$ of for which $\{\chi_i=0\}=\{x_i=0\}$, so that $\{\chi=0\}=\{(0,\ldots,0)\}\times\D^{m-k}$ in a neighborhood $\Omega_1$ of $0\in\D^m$. Let us fix $R=20$, let $\delta>0$ be given by Lemma \ref{lmMcM} and let us fix $0<\epsilon<\delta$. Let us denote by $(\mathcal{H}_j)$ the following assertion:

~

\begin{flushleft}
\textit{$(\mathcal{H}_j)$: \ \ There exists $\rho_j>0$ and a continuous embedding $\Phi_j:\Mand^j\times\D_{\rho_j}^{m-j}\longrightarrow\Omega_1$, and a continuous family $\{\varphi_{\zeta,x',l}:\p^1\longrightarrow\p^1\}_{(\zeta,x')\in\Mand^j\times\D_{\rho_j}^{m-j},1\leq l\leq j}$ of $(1+O(\epsilon))$-quasi-conformal homeomorphisms for which
\begin{enumerate}
\item  For any $1\leq l\leq j$, $t\in\D_{\rho_j}^{m-j}$, $\zeta_1,\ldots,\zeta_{l-1},\zeta_{l+1},\ldots,\zeta_j\in\Mand^{j-1}$ the map
\begin{center}
$\zeta\longmapsto\Phi_j(\zeta_1,\ldots,\zeta_{l-1},\zeta,\zeta_{l+1},\ldots,\zeta_j,t)$
\end{center}
is locally $1/(1+O(\epsilon))$-bih\"older continuous. Moreover, the h\"older constants are independant of $t$, $\zeta_1,\ldots,\zeta_{l-1},\zeta_{l+1},\ldots,\zeta_j$ and 
\begin{center}
$\Phi_j(\{\zeta_1,\ldots,\zeta_{i-1}\}\times \Mand\times\{\zeta_{i+1},\ldots,\zeta_j\})\subset\{a_1,\ldots,a_{i-1}\}\times\C\times\{a_{j+1},\ldots,a_m\}$
\end{center}
for some $a_i\in \C$, $i\neq l$, depending only on $\Phi_j$, $\zeta_i$, $i\neq l$ and $t\in\D^{m-j}_{\rho_j}$.
\item $\Phi_j$ is holomorphic on $(\mathring{\Mand})^j\times\D_{\rho_j}^{m-j}$,
\item For any $\zeta\in\Mand^j$, the set $\Phi_j(\{\zeta\}\times\D_{\rho_j}^{m-j})$ is a holomorphic graph of the form 
\begin{center}
$\Phi_j(\{\zeta\}\times\D_{p_j}^{m-j})=\big\{x_1=u_1(x'),\ldots,x_j=u_j(x'), \ x'\in\D^{m-j}_{\rho_j}\big\}$.
\end{center}
\item for $1\leq l\leq j,$ there exists $n_l\geq1$ such that  $\varphi_{\zeta,t,l}\circ p_{\zeta_l}=f_{\Phi(\zeta,t)}^{\circ n_l}\circ \varphi_{\zeta,t,l}$ on $\K_{\zeta_l}$ and the conjugacy is hybrid.
\end{enumerate}}
\end{flushleft}

\par We want to prove $(\mathcal{H}_j)$ by finite induction on $j$. Then, to conclude the proof of the Theorem, it only remains to justify the assertion $(3)$ of the Theorem. Let us begin with proving $(\mathcal{H}_1)$. To this aim, let us set
\begin{center}
$\Lambda_1\pe\{\chi_2=\ldots=\chi_k=0\}\cap\{x_{k+1}=\cdots=x_m=0\}=\{x\in\Omega_1\ |\ x_2=\ldots=x_m=0\}$.
\end{center}
Since $\chi$ is a local submersion at $0$, $\chi_1\not\equiv0$ on $\Lambda_1$. By \cite[Lemma 3.1]{Article1}, the critical point $c_1$ is thus active at $0$ in $\Lambda_1$. Since $c_1(0)$ is preperiodic under iteration of $f_0$, there exists $n\geq1$ such that $f^{\circ n}_0(c_1(0))$ is a periodic point of $f_0$. Moreover, it is a repelling periodic point. Up to multiplying $n$ by the period of $f^{\circ n}_0(c_1(0))$, we also may assume that $f^{\circ 2n}_0(c_1(0))=f^{\circ n}_0(c_1(0))$, i.e. that $f_0^{\circ n}(c_1(0))$ is a repelling fixed point for $f_0^{\circ n}$. By a Theorem of McMullen (see \cite[Theorem 3.1]{McMullen3}), there exists an integer $n_1\geq n$ and a coordinate change on $\p^1$, such that in this coordinate $c_1\equiv0$ on $\Lambda_1$ and
\begin{eqnarray}
f_\lambda^{\circ n_1}(z)=z^2+\zeta+h(z,\zeta),
\label{quadraticlike}
\end{eqnarray}
whenever $z,\zeta\in\D(0,2R)$, with $\sup|h(z,\zeta)|\leq \epsilon/2$ and $\lambda=\psi_1(\zeta)\pe t_1(1+\gamma_1\zeta)\in\Lambda_1$ and $0<|t_1|,|\gamma_1|<\epsilon$. Therefore, for $x'\in\D^{m-1}$ close enough to $0'$, in the coordinate given by Theorem 3.1 of \cite{McMullen3}, the map $f_\lambda$ satisfies \eqref{quadraticlike} for $z,\zeta\in\D(0,R)$, with $\sup|h(z,\zeta)|\leq \epsilon$ and $\lambda=\psi_1(\zeta)\pe t_1(1+\gamma_1\zeta)+x'\in\Lambda_1+x'$.  This means that there exists a family of quadratic-like maps $(f_\lambda^{\circ n_1})_{\lambda\in\psi_n(\D(0,R))\times\D_{\rho_1}^{m-1}}$ for some $\rho_1>0$ parametrized by the open set $\psi_1(\D(0,R))\times\D_{\rho_1}^{m-1}$ of $\D^m$. The existence of a surjective map
\begin{center}
$\phi_1:M_{\psi_1(\D(0,R))\times\D_{\rho_1}^{m-1}}\longrightarrow \Mand$
\end{center}
follows from Theorem \ref{tmDH}. Let us now set: 
\begin{eqnarray*}
\Psi_1:M_{\psi_n(\D(0,R))\times\D_{\rho_1}^{m-1}} & \longrightarrow & \Mand\times\D_{\rho_1}^{m-1}\\
\lambda & \longmapsto & (\phi_1(\lambda),\lambda_2,\ldots,\lambda_m).
\end{eqnarray*}
By Lemma \ref{lmMcM}, the map $\Psi_1|_{M_{\psi_1(\D(0,R))+x'}}:M_{\psi_1(\D(0,R))+x'}\longrightarrow\Mand\times\{x'\}$ is an homeomorphism which is the restriction of a $(1+O(\epsilon))$-quasiconformal map, for any $x'\in\D_{\rho_1}^{m-1}$. The assertions $(1)-(4)$ of $(\mathcal{H}_1)$ are then satisfied by $\Phi_1\pe\Psi_1^{-1}$, according to Theorem \ref{tmDH}.

~

\par We now assume that for $1\leq j\leq k-1$, we have already established assertion $(\mathcal{H}_j)$. Let us consider $\zeta^{(j)}\in(\partial\Mand)^j$ be such that the critical point of $z^2+\zeta_i^{(j)}$ is preperiodic to a repelling cycle and let us set
\begin{center}
$\Lambda_{j+1}\pe\Phi_j(\{\zeta^{(j)}\}\times\D_{\rho_j}^{m-j})\cap\{x_{j+2}=\cdots=x_m=0\}$
\end{center}
and let $\lambda^{(j+1)}\in\Lambda_{j+1}\cap\{x_{j+1}=0\}$. The critical points $c_{j+2},\ldots,c_k$ are passive in the family $\Lambda_{j+1}$ and, by the assumption $(4)$ of the induction hypothesis $(\mathcal{H}_j)$, up to reordering the critical points, we can assume that the critical points $c_1,\ldots,c_j$ are passive in the family $(f_t)_{t\in\Lambda_{j+1}}$. In addition, by assumption $(3)$ of $(\mathcal{H}_j)$, the set $\Lambda_{j+1}$ is of the form
\begin{center}
$\Lambda_{j+1}=\{x_1=u_1(x'),\ldots,x_j=u_j(x'),\ x'\in\D_{\rho_j}^{m-j}\}\cap\{x_{j+1}=\cdots=x_k=0\}$.
\end{center}
Therefore, the analytic sets $\Lambda_{j+1}$ and $\{x_{j+1}=0\}$ intersect properly at $\lambda^{(j+1)}$. Therefore, by \cite[Lemma 3.1]{Article1}, the critical point $c_{j+1}$ is active at $\lambda^{(j+1)}$ in $\Lambda_{j+1}$. Using again \cite[Theorem 3.1]{McMullen3},  we find an integer $n_{j+1}\geq1$ and a coordinate change on $\p^1$, such that in this coordinate $c_{j+1}\equiv0$ on $\Lambda_{j+1}$ and $f_t^{\circ n_{j+1}}(z)=z^2+\zeta+h(z,\zeta)$, whenever $z,\zeta\in\D(0,2R)$, with $\sup|h(z,\zeta)|\leq \epsilon/2$ and $t=\psi_{j+1}(\zeta)\pe t_{j+1}(1+\gamma_{j+1}\zeta)\in\Lambda_{j+1}$ and $0<|t_{j+1}|,|\gamma_{j+1}|<\epsilon$. We then proceed as in the previous step to find $0<r\leq\rho_j$ and to build a continuous injection
\begin{center}
$\Psi_{j+1}:\D_r^j\times\Mand\times\D_r^{m-j-1}\longrightarrow\D^m(\Phi_j(\zeta^{(j)},0'),r)$
\end{center}
satisfying $(\mathcal{H}_1)$. In particular, for any $\zeta\in\Mand$, the set $\Psi_{j+1}(\D_r^j\times\{\zeta\}\times\D_r^{m-j-1})$ is a holomorphic graph of the form 
\begin{center}
$\Psi_{j+1}(\D_r^j\times\{\zeta\}\times\D_r^{m-j-1})=\big\{x_{j+1}=u_{j+1}(x',x''), \ x'\in\D^j_r,x''\in\D^{m-j-1}_r\big\}$.
\end{center}

\par We now may construct the map $\Phi_{j+1}$, using $\Phi_j$ and $\Psi_{j+1}$. By a classical result of Douady and Hubbard (see \cite{DH}, see also \cite[Theorem 4.1]{McMullen3}), there exist $(1+\epsilon)$-quasiconformal embeddings $\phi_i:\Mand\longrightarrow\Mand$ whose images are, respectively, contained in arbitrary small neighborhoods of $\zeta_i^{(j)}$. Therefore, the maps $\phi_i$ ca be chosen so that
\begin{center}
$\Phi_j\big(\phi_1(\Mand)\times\cdots\times\phi_j(\Mand)\times\D_{\rho_j}^{m-j}\big)\cap\Psi_{j+1}(\D_r^j\times\{\zeta_{j+1}\}\times\{0\})\Subset\Psi_{j+1}(\D_r^j\times\{\zeta_{j+1}\}\times\{0\})$
\end{center}
for any $\zeta_{j+1}\in\Mand$. By continuity of $\Psi_{j+1}$, we thus can find $0<\rho_{j+1}\leq r$ such that
\begin{center}
$\Phi_j\big(\phi_1(\Mand)\times\cdots\times\phi_j(\Mand)\times\D_{\rho_j}^{m-j}\big)\cap\Psi_{j+1}(\D_{\rho_{j+1}}^j\times\{\zeta\}\times\{x'\})\Subset\Psi_{j+1}(\D_{\rho_{j+1}}^j\times\{\zeta\}\times\{x'\})$,
\end{center}
for any $(\zeta_{j+1},x')\in\Mand\times\D_{\rho_{j+1}}^{m-j-1}$. The hypothesis $(3)$ of $(\mathcal{H}_j)$ guaranties that for any $\zeta_1,\ldots,\zeta_{j+1}\in\Mand$ and any $x'\in\D_{\rho_{j+1}}^{m-j-1}$, the intersection
\begin{center}
$\Phi_j\big(\{(\phi_1(\zeta_1),\ldots,\phi_j(\zeta_j))\}\times\D_{\rho_j}\times\{x'\}\big)\cap\Psi_{j+1}(\D_{\rho_{j+1}}^j\times\{\zeta\}\times\{x'\})$
\end{center}
is reduced to one point.

 We define $\Phi_{j+1}(\zeta,x')$ as this unique intersection point. The properties of $\Phi_j$ and $\Psi_{j+1}$ respectively given by $(\mathcal{H}_j)$ and $(\mathcal{H}_1)$ directly imply that the map $\Phi_{j+1}$ satisfies the assertions $(2)$, $(3)$ and $(4)$ of $(\mathcal{H}_{j+1})$. To conclude, it remains to remark that, by the regularity properties of $\Phi_j$ and $\Psi_{j+1}$, the map $\Phi_{j+1}$ obviously satisfies $(1)$.
 
~

\par We have shown that $\Phi$ exists and satisfies $(1)$, $(2)$ and $(4)$. It remains to justify the fact that $\Phi$ satisfies $(3)$. First, let us remark that assumption $(1)$ of $(\mathcal{H}_k)$ combined with Lemma \ref{lmdimH} implies that for any $t\in\D^{m-k}$, the map $\Phi(\cdot,t):\Mand^k\longrightarrow\Omega_1$ is locally $1/(1+O(\epsilon))$-bih\"older. Let now $\zeta\in(\partial\Mand)^k$ and let $\rho>0$ be such that $\Phi(\cdot,t)$ is $1/(1+O(\epsilon))$-bih\"older on $\D^k(\zeta,\rho)$. Lemma \ref{lmdimH} and \cite[Theorem A]{shishikura2} give
\begin{eqnarray*}
\dim_H\big(\Phi\big((\partial\Mand)^k\cap\D^k(\zeta,\rho),t\big)\big) & \geq & (1+O(\epsilon))\dim_H((\partial\Mand)^k\cap\D^k(\zeta,\rho))\\
& \geq & (1+O(\epsilon))\sum_{j=0}^k\dim_H((\partial\Mand)\cap\D(\zeta_j,\rho))\\
& \geq & 2k(1+O(\epsilon)).
\end{eqnarray*}
Lemma \ref{lmdimHMcMullen} and assertion $(3)$ of $(\mathcal{H}_k)$ then state that for almost every $t\in\D^{m-k}$,
\begin{eqnarray*}
\dim_H\big(\Phi\big((\partial\Mand)^k\times\D^{m-k}\big)\big) & \geq & 2(m-k)+\dim_H\big(\Phi\big((\partial\Mand)^k\times\{t\}\big)\big)\\
& \geq & 2(m-k)+2k(1+O(\epsilon))=2m-O(\epsilon),
\end{eqnarray*}
which ends the proof.

\subsection{A consequence: Theorem \ref{mainthmMcM}.}\label{sec:tm2}
\par We now prove that the homeomorphically embedded copies of $(\partial\Mand)^k\times\D^{\dim\Lambda-k}$ given by Theorem \ref{propmuniv} are contained in the support of the bifurcation currents. As a consequence, we obtain optimal Hausdorff dimension estimates for the supports of the bifurcation currents. Theorem \ref{propmuniv} combined with \cite[Theorem 6.2]{Article1} yields the following key proposition.

\begin{proposition}
Let $(f_\lambda)_{\lambda\in\Lambda}$ be a holomorphic family of degree $d$ rational maps. Assume that $c_1,\ldots,c_k$ are marked simple critical points and denote by $T_1,\ldots,T_k$ their respective bifurcation currents. Assume that $k\leq m=\dim\Lambda$ and $T_1\wedge\cdots\wedge T_k\neq0$. Then, for any $\epsilon>0$, the homeomorphic embeddings of the set $(\partial\Mand)^k\times\D^{m-k}$ given by Theorem \ref{propmuniv} are contained in $\supp(T_1\wedge\cdots\wedge T_k)$.
\label{cormuniv}
\end{proposition}

\begin{proof}
Consider a dense sequence $\zeta_j\subset\partial\Mand$ for which $0$ is preperiodic to a repelling cycle for $z^2+\zeta_j$. Since $\partial\Mand$ is the bifurcation locus of the family $(z^2+\zeta)_{\zeta\in\C}$, the existence of such a sequence is just an straight foward consequence of Montel's Theorem (see for example \cite[Lemma 2.3]{favredujardin} or \cite[Lemma 2.1]{McMullen3}). Set $\mathbf{j}\pe(j_1,\ldots,j_k)$ and $\zeta_{\mathbf{j}}\pe(\zeta_{j_1},\ldots,\zeta_{j_k})$. Let $\Phi$ be the embedding given by Theorem \ref{propmuniv}. Since the set $\{(\zeta_{\mathbf{j}},x')\in(\partial\Mand)^k\times\D^{m-k} \ | \ \mathbf{j}\in(\Z_+)^k\}$ is dense in $(\partial\Mand)^k\times\D^{m-k}$, it is sufficient to show that $\Phi(\zeta_{\mathbf{j}},x')\in\supp(T_1\wedge\cdots\wedge T_k)$ for all $\mathbf{j}\in(\mathbb{Z}_+)^k$ and all $x'\in\D^{m-k}$. By item (4) of Theorem \ref{propmuniv}, the critical points $c_1,\ldots,c_k$ fall onto repelling cycles at $\Phi(\zeta_{\mathbf{j}},x')$ for any $x'\in\D^{m-k}$. Since that $c_1,\ldots,c_k$ fall properly onto repelling cycles for any $\mathbf{j}\in(\Z_+)^k$. \cite[Theorem 6.2]{Article1} then states that $\Phi(\zeta_{\mathbf{j}},x')\in\supp(T_1\wedge\cdots\wedge T_k)$.
\end{proof}

\begin{proof}[Proof of Theorem \ref{mainthmMcM}]
This is a direct consequence of Theorem \ref{propmuniv} and Proposition \ref{cormuniv}.
\end{proof}

\subsection{Hausdorff dimension of the support of bifurcation currents.}

\par To end this section, we want to underline the fact that Theorem \ref{propmuniv}, Proposition \ref{cormuniv} and the work \cite{higher} of Dujardin directly give Hausdorff dimension estimates for the support of $T_1\wedge \cdots\wedge T_k$.

\begin{proposition}
Let $(f_\lambda)_{\lambda\in\Lambda}$ be a holomorphic family of degree $d$ rational maps. Assume that $c_1,\ldots,c_k$ are marked simple critical points and denote by $T_1,\ldots,T_k$ their respective bifurcation currents. Assume that $k\leq m=\dim\Lambda$ and $T_1\wedge\cdots\wedge T_k\neq0$. Then, for any $\epsilon>0$, the homeomorphic embeddings of the set $(\partial\Mand)^k\times\D^{m-k}$ of dimension at least $2m-\epsilon$ given by Theorem \ref{propmuniv} are dense in $\supp(T_1\wedge\cdots\wedge T_k)$.
\label{cormuniv2}
\end{proposition}

\begin{proof}
Let $\lambda_0\in\supp(T_1\wedge\cdots\wedge T_k)$ and $\epsilon>0$. By \cite[Theorem 0.1]{higher}, there exists a sequence $\lambda_n\rightarrow\lambda_0$ such that $c_1,\ldots,c_k$ fall transversely onto repelling cyles at $\lambda_n$. Let $n\geq 1$ be such that $\lambda_n\in\B(\lambda_0,\epsilon)$. Then, by Theorem \ref{propmuniv} and Proposition \ref{cormuniv}, there exists an embedding
$$\Phi:(\partial\Mand)^k\times\D^{m-k}\longrightarrow\B(\lambda_0,\epsilon)\cap\supp(T_1\wedge\cdots\wedge T_k)$$
with $\dim_H(\Phi((\partial\Mand)^k\times\D^{m-k}))\geq 2m-\epsilon$.
\end{proof}

Let $(X,d)$ be a metric space. Recall that $E\subset X$ is said to \emph{homogeneous} if $\dim_H(E\cap U)=\dim_H(E)$ for all open set $U\subset X$ with $U\cap E\neq\emptyset$. As a consequenc of Proposition \ref{cormuniv2}, we get the following.

\begin{corollary}
Let $(f_\lambda)_{\lambda\in\Lambda}$ be a holomorphic family of degree $d$ rational maps. Assume that $c_1,\ldots,c_k$ are marked simple critical points and denote by $T_1,\ldots,T_k$ their respective bifurcation currents. Then either
\begin{itemize}
\item $T_1\wedge\cdots\wedge T_k=0$, or,
\item $\supp(T_1\wedge\cdots\wedge T_k)$ is homogeneous and has maximal Hausdorff dimension $2m$.
\end{itemize}
\label{cordimbif}
\end{corollary}

\par We are now in position to prove Theorem \ref{tmdim}.

\begin{proof}[Proof of Theorem \ref{tmdim}]
Let $k\geq1$ be such that $T_\bif^k\neq0$. Up to taking a finite branched covering of the family $(f_\lambda)_{\lambda\in\Lambda}$, we can assume that it has marked critical points $c_1,\ldots,c_{2d-2}$. If $T_i$ is the bifurcation current of the critical points $c_i$, \eqref{Demarco2} gives
\begin{eqnarray}
\supp(T_\bif^k)=\bigcup_{1\leq j_1<\cdots<j_k\leq2d-2}\supp\left(\bigwedge_{i=1}^kT_{j_i}\right).
\label{egsupp}
\end{eqnarray}
Let us now set
\begin{center}
$\mathcal{C}_{i,j}\pe\{\lambda\in\Lambda \ | \ c_j(\lambda)=c_i(\lambda)\}$
\end{center}
for $1\leq i\neq j\leq 2d-2$. By assumption, $\mathcal{C}_{i,j}$ is a complex hypersurface of $\Lambda$. Let $\Lambda_1\pe\Lambda\setminus\bigcup_{i\neq j}\mathcal{C}_{i,j}$. Then the family $(f_\lambda)_{\lambda\in\Lambda_1}$ is a family of degree $d$ rational maps with simple marked critical points. The key of the proof is the following lemma.

\begin{lemma}
Let $\B\subset\Lambda_1$ be an open ball and let
\begin{center}
$m\pe\max\{1\leq j\leq 2d-2 \ | \ T_\bif^j\neq0$ in $\B\}$.
\end{center}
Then $\B\cap\supp(T_\bif^{m-1})\setminus\supp(T_\bif^m)\neq\emptyset$.
\label{lminter}
\end{lemma}

To finish the proof of Theorem \ref{tmdim}, it suffices to show that $\Lambda_1\cap\supp(T_\bif^k)\setminus\supp(T_\bif^{k+1})\neq\emptyset$ and then to apply Corollary \ref{cordimbif} in any ball $\B\subset\Lambda_1$ such that $\B\cap\supp(T_\bif^k)\subset\supp(T_\bif^k)\setminus\supp(T_\bif^{k+1})$.
\par By Lemma \ref{lminter}, if $m=\max\{j\leq 2d-2$ / $T_\bif^j\neq0$ on $\Lambda_1\}$, there exists $\lambda_0\in\supp(T_\bif^{m-1})\setminus\supp(T_\bif^m)$. If $\B_1\subset\Lambda_1$ is a small enough ball centered at $\lambda_1$, one has $\supp(T_\bif^m)\cap\B_1=\emptyset$, then applying again Lemma \ref{lminter}, we find $\lambda_1\in\B_1\cap\supp(T_\bif^{m-2})\setminus\supp(T_\bif^{m-1})$. In $m-k+1$ steps, we find $\lambda_{m-k+1}\in\supp(T_\bif^k)\setminus\supp(T_\bif^{k+1})$.
\end{proof}

\begin{proof}[Proof of Lemma \ref{lminter}]
This is a consequence of \cite[Theorem 0.1]{higher}. Let $\lambda_0\in\supp(T_\bif^m)\cap\B$, then by \eqref{egsupp}, there exists $1\leq j_1<\cdots<c_{j_m}\leq 2d-2$ such that $\lambda_0\in\supp(T_{j_1}\wedge\cdots\wedge T_{j_m})$. By Theorem \ref{tmduj}, there exists $\lambda_1\in\B$ such that $c_{j_1},\ldots,c_{j_m}$ fall transversely onto repelling cycles (see Definition \ref{deftransvers}). Let now $n_i,k_i\geq1$ be such that
\begin{center}
$\lambda_1\in X_i\pe\{\lambda\in\B \ | \ f_\lambda^{\circ n_i}(c_{j_i}(\lambda))=f_\lambda^{\circ (n_i+k_i)}(c_{j_i}(\lambda))$ and $f_\lambda^{\circ n_i}(c_{j_i}(\lambda))$ is repelling$\}$
\end{center}
for any $1\leq i\leq m$. By \cite[Lemma 3.1]{Article1}, the critical point $c_{j_m}$ is active at $\lambda_1$ in $X_{j_1}\cap\cdots\cap X_{j_{m-1}}$. By Montel's Theorem, there exists $\lambda_2\in X_{j_1}\cap\cdots\cap X_{j_{m-1}}$ such that $c_{j_m}(\lambda_2)$ is a periodic point of $f_{\lambda_2}$. Therefore, there exists $\B_1\Subset\B$ a ball centered at $\lambda_2$ such that $c_{j_m}$ is passive on $\B_1$ and $T_{j_1}\wedge\cdots\wedge T_{j_{m-1}}\neq0$ on $\B_1$.
\par Assume now that $T_\bif^m\neq0$ on $\B_1$. By the same procedure, we can find $j_m'\neq j_m$ and a ball $\B_2\Subset\B_1$ such that $c_{j_m'}$ is passive on $\B_2$ and $T_\bif^{m-1}\neq0$ on $\B_2$. In finitely many steps, we find a ball $\B'\Subset\B$ with 
\begin{enumerate}
\item $2d-2-m+1$ critical points are passive on $\B'$,
\item $T_\bif^{m-1}\neq0$ on $\B'$, i.e. $\supp(T_\bif^{m-1})\cap\B'\neq\emptyset$.
\end{enumerate}
Since item (1) gives $\supp(T_\bif^{m-1})\cap\B'\subset\supp(T_\bif^{m-1})\setminus\supp(T_\bif^m)$, the proof is complete.
\end{proof}

\section{Higher bifurcation currents and neutral cycles.}

\par One of the interesting facts provided by the work \cite{MSS} of Ma\~n\'e, Sad and Sullivan and the work \cite{Lyubich} of Lyubich is the existing link between the existence of a non-persitent neutral cycle and the non-persistent preperiodicity of a critical point. Namely, they show that in any holomorphic family $(f_\lambda)_{\lambda\in\Lambda}$ of degree $d$ rational maps, the closure in $\Lambda$ of the set of parameters $\lambda_0$ for which $f_{\lambda_0}$ possesses a non-persistent neutral cycle coincides with the closure in $\Lambda$ of the set of parameters $\lambda_0$ for which one critical point of $f_{\lambda_0}$ is non-persistently preperiodic to a repelling cycle. In this section, we want to establish an analogous result for higher bifurcation loci.

\subsection{In the space $\rat_d^{cm}$ of critically marked degree $d$ rational maps.}\label{sectionratd}

\par We refer to \cite[Section 1.2]{buffepstein} for a description of the set $\rat_d^{cm}$ of critically marked rational maps. The space $\rat_d^{cm}$ is a quasiprojective variety of dimension $2d+1$, which is an algebraic finite branched cover of $\rat_d$. The degree of the natural projection $\pi:\rat_d^{cm}\longrightarrow\rat_d$ depends only on $d$. Moreover, there exists $2d-2$ holomorphic maps $c_1,\ldots,c_{2d-2}:\rat_d^{cm}\longrightarrow\p^1$ such that $C(f)=\{c_1(f),\ldots,c_{2d-2}(f)\}$, where the critical points are counted with multiplicity. Recall that a holomorphic family of rational maps is said \emph{algebraic} if its parameter space is an algebraic variety, and that it is said \emph{stable} if its bifurcation locus is empty. In what follows, we will need the following lemma (see \cite[Lemma 2.1]{McMullen4}).

\begin{lemma}[McMullen]
Any stable algebraic family of degree $d$ rational maps is either trivial or all its members are postcritically finite.
\label{lmalgebraic}
\end{lemma}

\par Recall that for $n\geq1$ and $w\in\C\setminus\{1\}$, we denoted by $\Per_n(w)$ the set of all rational maps having a cycle of multiplier $w$ and exact period $n$ (see section \ref{sectionPern}). In the quasiprojective variety $\rat_d^{cm}$, the set $\Per_n(w)$ is an algebraic hypersurface. Let $k\geq2$, for $\Theta_k\pe(\theta_1,\ldots,\theta_k)\in(\R\setminus\Z)^k$ and $N_k\pe(n_1,\ldots,n_k)\in(\Z_+)^k$ we define the set $\Per^k_{N_k}(\Theta_k)$ as 
\begin{center}
$\Per^k_{N_k}(\Theta_k)\pe\{f\in\rat_d^{cm} \ | \ f$  has $k$ distinct neutral cycles of respective multipliers $e^{2i\pi\theta_1},\ldots,e^{2i\pi\theta_k}$ and respective period $n_1,\ldots, n_k\}$. 
\end{center}
The set $\Per^k_{N_k}(\Theta_k)$ is a subvariety of $\bigcap_{1\leq j\leq k}\Per_{n_j}(e^{2i\pi\theta_j})$.

~

\begin{lemma}
Let $k\geq2$, $\Theta_k=(\theta_1,\ldots,\theta_k)\in(\R\setminus\Z)^k$ and $N_k=(n_1,\ldots,n_k)\in(\Z_+)^k$. If $\Per_{N_k}^k(\Theta_k)\neq\emptyset$, then any irreducible component of the algebraic set $\Per_{N_k}^k(\Theta_k)$ has codimension $k$ in $\rat_d^{cm}$.
\label{lemmaneutral}
\end{lemma}

\begin{proof}
Let $\Gamma$ be an irreducible component of $\Per_{N_k}^k(\Theta_k)$. Let us first treat the case $k=2d-2$. If $\textup{codim}\ \Gamma<2d-2$, the family $\Gamma$ is a non-trivial algebraic family of rational maps, since $\dim\Gamma>3$ and it is stable, by the Fatou-Shishikura inequality. Lemma \ref{lmalgebraic} asserts that the family $\Gamma$ is a family of postcritically finite rational maps. This is impossible, since postcritically finite rational maps only have repelling or attracting cycles. This implies that $\textup{codim}\ \Gamma=2d-2$.
\par Assume now that $k<2d-2$. Then the family $\Gamma$ is not stable. Indeed, if we assume that $\Gamma$ is stable, Lemma \ref{lmalgebraic} again implies that $\Gamma$ is a family of postcritically finie rational maps. Therefore, there exist $\theta_{k+1}\in\R\setminus\Z\cup\{\theta_1,\ldots,\theta_k\}$, an integer $n_{k+1}$ and a map $f_1\in\Gamma\cap\Per_{n_{k+1}}(e^{2i\pi\theta_{k+1}})$. We thus reduce to proving that any irreducible component of $\Per^{k+1}_{N_{k+1}}(\Theta_{k+1})$ has codimension $k+1$, which in finitely many steps boils down to the case $k=2d-2$.
\end{proof}

\par Let $1\leq k\leq 2d-2$. For $\Theta_k=(\theta_1,\ldots,\theta_k)\in(\R\setminus\Z)^k$, recall that we have set
\begin{center}
$\mathcal{Z}_k(\Theta_k)=\displaystyle\bigcup_{N_k\in(\Z_+)^k}\Per_{N_k}^k(\Theta_k)$.
\end{center}
Recall also that we denoted by $\textup{Prerep}(k)$ the set of rational maps having $k$ prerepelling critical points. We still denote by $T_\bif^k$ the $k$-th bifurcation current of the family $\rat_d^{cm}$ which may be defined by 
$$T_\bif^k\pe \pi^*\left((dd^cL)^k\right)=\left(dd^c(L\circ \pi)\right)^k~.$$
Our main result of the present section may be stated as follows:

~

\begin{theorem}
Let $1\leq k\leq 2d-2$ and let $\Theta_k=(\theta_1,\ldots,\theta_k)\in (\R\setminus\Z)^k$. Then in $\rat_d^{cm}$
\begin{center}
$\supp(T_\bif^k)=\displaystyle\overline{\mathcal{Z}_k(\Theta_k)}=\overline{\pi^{-1}(\textup{Prerep}(k))}$,
\end{center}
\label{MSSk}
\end{theorem}

\begin{proof}
By \cite[Theorem 1]{higher}, we already know that $\overline{\pi^{-1}\textup{Prerep}(k)}=\supp(T_\bif^k)$. The first step of the proof consists in showing that $\mathcal{Z}_k(\Theta_k)$ is not empty and $\supp(T_\bif^k)\subset\overline{\mathcal{Z}_k(\Theta_k)}$. After that, we show that, when $\Per_{N_k}^k(\Theta_k)\neq \emptyset$, it is contained in $\supp(T_\bif^k)$.

~

\par Since $T_\bif^{2d-2}\neq0$, the current $T_{j_1}\wedge\cdots\wedge T_{j_k}$ is non-zero for any $j_1<\cdots<j_k$ and Proposition \ref{cormuniv} implies that there exists a family of homeomorphic embeddings
\begin{center}
$\Phi_n:(\partial\Mand)^k\times\D^{2d+1-k}\longrightarrow\supp(T_{j_1}\wedge\cdots\wedge T_{j_k})$
\end{center}
which images are dense in $\supp(T_{j_1}\wedge\cdots\wedge T_{j_k})$. Let $\zeta_1,\ldots\zeta_k\in\partial\Mand$ be such that $z^2+\zeta_j$ has a cycle of multiplier $e^{2i\pi\theta_j}$. The conjugacy given by Theorem \ref{propmuniv} being hybrid, Lemma 3 of \cite{buffhenriksen} ensures that the map $f_{\Phi_n(\zeta,0)}$ has $k$ distinct neutral cycles of respectives multipliers $e^{2i\pi\theta_1},\ldots,e^{2i\pi\theta_k}$ and thus $\Per_{N_k}^k(\Theta_k)\neq\emptyset$ for some $N_k\in(\Z_+)^k$. Moreover, Corollary \ref{density} asserts that, for any $1\leq j\leq k$, the set of parameters $\zeta\in\partial\Mand$ for which $z^2+\zeta$ has a cycle of multiplier $e^{2i\pi\theta_j}$ is dense in $\partial\Mand$. Therefore, $\supp(T_{j_1}\wedge\cdots\wedge T_{j_k})\subset\overline{\mathcal{Z}_k(\Theta_k)}$, for any $j_1<\cdots<j_k$. By \eqref{Demarco2}, this implies $\supp(T_\bif^k)\subset\overline{\mathcal{Z}_k(\Theta_k)}$.

~

\par It thus remains to prove that $\Per_{N_k}^k(\Theta_k)\subset\supp(T_\bif^k)$, as soon as $\Per_{N_k}^k(\Theta_k)\neq\emptyset$. To this aim, we set for $m>n\geq1$ and $1\leq j\leq 2d-2$:
\begin{center}
$\textup{Prerep}_j(n,m)=\{f\in\rat_d^{cm} \ | \ f^{\circ n}(c_j(f))=f^{\circ m}(c_j(f))$ and $f^{\circ(m-n)}(c_j(f))$ is repelling$\}$.
\end{center}
We proceed by induction. Let $N_k=(n_1,\ldots,n_k)\in(\Z_+)^k$ be such that $\Per_{N_k}^k(\Theta_k)\neq\emptyset$ and let $f_0\in\Per_{N_k}(\Theta_k)$. By Lemma \ref{lemmaneutral}, $f_0$ has a non-persistent cycle of multiplier $e^{2i\pi\theta_k}$ in the family $\Per_{N_{k-1}}^{k-1}(\Theta_{k-1})$. Ma\~n\'e-Sad-Sullivan's Theorem asserts that $f_0$ is a bifurcation parameter in the family $\Per_{N_k}(\Theta_k)$. Therefore, by Montel's Theorem, there exists $f_1\in\Per_{N_{k-1}}^{k-1}(\Theta_{k-1})$ aribrarily close to $f_0$ so that $f_1$ has one critical point preperiodic to a repelling cycle, i.e.
\begin{center}
$f_1\in\Per_{N_{k-1}}^{k-1}(\Theta_{k-1})\cap\textup{Prerep}_{j_1}(n_1,m_1)$
\end{center}
for some $1\leq j_1\leq 2d-2$ and $m_1>n_1\geq1$ and $\Per_{N_{k-1}}^{k-1}(\Theta_{k-1})\cap\textup{Prerep}_j(n_1,m_1)$ has codimension $k$. Assume now that we already have found
\begin{center}
$f_j\in \displaystyle\bigcap_{1\leq i\leq j}\textup{Prerep}_{j_i}(n_i,m_i)\cap\Per_{N_{k-j}}^{k-j}(\Theta_{k-j})$
\end{center}
arbitrarily close to $f_0$ and that $\textup{codim}\bigcap_{1\leq i\leq j}\textup{Prerep}_{j_i}(n_i,m_i)\cap\Per_{N_{k-j}}^{k-j}(\Theta_{k-j})=k$. Then, the map $f_j$ has a non-persistent neutral cycle of multiplier $e^{2i\pi\theta_{k-j}}$ in the family
\begin{center}
$X_j\pe\displaystyle\bigcap_{1\leq i\leq j}\textup{Prerep}_{j_i}(n_i,m_i)\cap\Per_{N_{k-j-1}}^{k-j-1}(\Theta_{k-j-1})$.
\end{center}
Remark that the fact that a periodic point is repelling is an open condition. Thus, using again Montel's Theorem, we find integers $m_{j+1}>n_{j+1}\geq 1$ and
\begin{center}
$f_{j+1}\in\textup{Prerep}_{j_{j+1}}(n_{j+1},m_{j+1})\cap X_j$
\end{center}
arbitrarily close to $f_j$. Moreover, $\textup{codim}\ \textup{Prerep}_{j_{j+1}}(n_{j+1},m_{j+1})\cap X_j=k$.
\par Iterating this process $k$ times, we find $f_k$ arbitrarily close to $f_0$ at which $k$ critical points fall properly onto repelling cycles. Theorem 6.2 of \cite{Article1} states that, under these conditions, the map $f_k$ belongs to the support of $T_\bif^k$. As $f_k$ can be taken as close to $f_0$ as we want, this concludes the proof.\end{proof}

\begin{proof}[Proof of Theorem \ref{maintheorem}]
 Recall that we denoted by $\pi:\rat_d^{cm}\longrightarrow\rat_d$ the natural projection, which is a finite branched covering. The projection $$\Pi:\rat_d^{cm}\longrightarrow\mathcal{M}_d$$
 which, to $f$ associates its class of conjugacy by M\"obius transformations, is a principal bundle on $\rat_d^{cm}\setminus V$, where $V$ is a proper subvariety of $\rat_d^{cm}$ (see e.g. \cite{BB1} page 226). Since the function $L\circ\pi:\rat_d^{cm}\longrightarrow\R$ is continuous, the current $(dd^c(L\circ \pi))^k$ doesn't give mass to pluripolar sets. Therefore, Theorem \ref{MSSk} implies that the set $\mathcal{Z}_k(\Theta_k)\setminus V$ is dense in $\supp((dd^c(L\circ\pi))^k)$. The conclusion follows, since $\Pi(\supp((dd^cL\circ\pi)^k))=\supp(T_\bif^k)$, where $T_\bif^k$ denotes the $k^{th}$-bifurcation current of the moduli space $\mathcal{M}_d$.
\end{proof}

\subsection{In the moduli space $\mathcal{P}_d$ of degree $d$ polynomials.}

In the present section, we want to give a simpler argument for the proof of Theorem \ref{MSSk} in the case of polynomial families. This argument relies on a fine control of the cluster set of the bifurcation locus at infinity. To this aim, we will use the following paramtrization of the moduli space $\mathcal{P}_d$ of all degree $d$ polynomials. For any $(c,a)=(c_1,\ldots,c_{d-2},a)\in\C^{d-1}$, we set $c=(c_1,\ldots,c_{d-2})$ and
\begin{center}
$\displaystyle P_{(c,a)}(z)\pe \frac{1}{d}z^d+\sum_{j=2}^{d-1}(-1)^{d-j}\sigma_{d-j}(c)\frac{z^j}{j}+a^d$,
\end{center}
where $\sigma_j(c)$ is the symmetric degree $j$ polynomial in $c_1,\ldots,c_{d-2}$. The critical points of the polynomial $P_{(c,a)}$ are $0,c_1,\ldots,c_{d-2}$ and are holomorphic functions of the parameter. This family has been introduced by Branner and Hubbard in \cite{BH} to prove the compactness of the connectedness locus of $\mathcal{P}_d$. It also has been used by Dujardin and Favre in \cite{favredujardin} and by Bassanelli and Berteloot to study the bifurcation currents in \cite{BB2}.
\par The parameter space $\C^{d-1}$ can be naturally compactified as $\p^{d-1}$ by the following natural injection:
\begin{center}
$(c,a)\in\C^{d-1}\longrightarrow[c:a:1]\in\p^{d-1}$.
\end{center}
Finally, we denote by $T_i$ the bifurcation current of the marked critical point $c_i$. Let us set $\mathcal{C}_d=\{(c,a)\in\C^{d-1} \ | \ \J_{c,a}$ is connected$\}$. We summarize the main properties of this parametrization in the following proposition (see \cite{BH}, \cite[\S 6]{favredujardin} and \cite[\S 4]{BB2}):

~

\begin{proposition}
\begin{enumerate}
\item The natural projection $\Pi:\C^{d-1}\longrightarrow\mathcal{P}_d$ is a degree $d(d-1)$ analytic branched cover,
\item The loci
$$\mathcal{B}_i\pe\{(c,a) \, | \, (P^{\circ n}_{(c,a)}(c_i))_{n\geq1}\text{ is bounded in }\C\}$$ accumulate at infinity of $\C^{d-1}$ in $\p^{d-1}$ on codimension $1$ algebraic sets $\Gamma_i$ of the hyperplane $\p_{\infty}=\p^{d-1}\setminus\C^{d-1}$,
\item The locus $\mathcal{C}_d$ is compact in $\C^{d-1}$ and, for any $0\leq i_1<\cdots<i_p\leq d-2$, the intersection $\Gamma_{i_1}\cap\cdots\cap\Gamma_{i_p}$ has codimension $p$ in $\p_{\infty}$.
\item The bifurcation measure $\mu_\bif\pe T_\bif^{d-1}$ is a finite positive measure and its support coincides with the Shilov boundary of $\mathcal{C}_d$. 
\end{enumerate}
\label{propPolyd}
\end{proposition}

\par For any $w\in\C$, the algebraic hypersurfaces $\Per_n(w)$ of $\C^{d-1}$ extend as algebraic hypersurfaces of $\p^{d-1}$. Moreover, for $w\in\overline{\D}$, the hypersurface $\Per_n(w)$ intersect the hyperplane at infinity $\p_\infty$ along the algebraic set $\bigcup_{0\leq j\leq d-2}\p_\infty\cap\mathcal{B}_j$, which has codimension $2$ in $\p^{d-1}$.

We use the same notations as in section \ref{sectionratd}. Let $1\leq k\leq d-1$, for $N_k=(n_1,\ldots,n_k)\in(\mathbb{Z}_+)^k$ and $\Theta_k=(\theta_1;\ldots,\theta_k)\in(\R\setminus\Z)^k$, we denote by $\Per_{N_k}^k(\Theta_k)$ the set of parameters $(c,a)\in\C^{d-1}$ s.t. $P_{(c,a)}$ has $k$ distinct neutral cycles of respective multipliers $e^{2i\pi\theta_j}$ and period $n_j$. The set $\Per^k_{N_k}(\Theta_k)$ is a subvariety of $\bigcap_{1\leq j\leq k}\Per_{n_j}(e^{2i\pi\theta_j})$. We also set
\begin{center}
$\mathcal{Z}_k(\Theta_k)\pe\displaystyle\bigcup_{N_k\in(\Z_+)^k}\Per_{N_k}^k(\Theta_k)$
\end{center}
and $\textup{Prerep}(k)\pe\{(c,a)\in\C^{d-1} \ | \ P_{(c,a)}$ has $k$ prereppelling critical points$\}$. In the present setting, Theorem \ref{MSSk} can be formulated as follows.

\begin{theorem}
Let $1\leq k\leq d-1$ and let $\Theta_k=(\theta_1,\ldots,\theta_k)\in (\R\setminus\Z)^k$. Then, in $\C^{d-1}$,
\begin{center}
$\supp(T_\bif^k)=\displaystyle\overline{\mathcal{Z}_k(\Theta_k)}=\overline{\textup{Prerep}(k)}$,
\end{center}
\label{MSSkpoly}
\end{theorem}

\par The proof is the same as in the case of the space $\rat_d^{cm}$. The only difference is in the proof of the following lemma.

\begin{lemma}
Let $k\geq2$, $\Theta_k=(\theta_1,\ldots,\theta_k)\in(\R\setminus\Z)^k$ and $N_k=(n_1,\ldots,n_k)\in(\Z_+)^k$. If $\Per_{N_k}^k(\Theta_k)\neq\emptyset$, then any irreducible component of the algebraic set $\Per_{N_k}^k(\Theta_k)$ has codimension $k$ in $\C^{d-1}$.
\label{lemmaneutralpoly}
\end{lemma}

\begin{proof}
Let $\Gamma$ be a non-empty irreducible component of $\Per_{N_k}^k(\Theta_k)$. Then, there exists irreducible components $H_i$ of $\Per_{n_i}(e^{2i\pi\theta_i})$, such that $\Gamma$ is a Zariski open set of $H_1\cap\cdots\cap H_k$. For any $1\leq i\leq k$, note that $H_i\subset\bigcup_j\mathcal{B}_j$. By Proposition \ref{propPolyd}, this implies that $\p_\infty\cap\bigcap_{1\leq i\leq k}H_i$ has codimension $k+1$. Since $\p_\infty$ has codimension $1$, we get $\textup{codim}\ H_1\cap\cdots\cap H_k=k$.
\end{proof}

\begin{remark}
Dujardin and Favre proved that for a $\mu_\bif-$generic polynomial $f$, all critical orbits are dense in $\J_f$ (see \cite[Corollary 11]{favredujardin}). Therefore, the copies of $(\partial\Mand)^{d-1}$ provided by Theorem \ref{mainthmMcM} have zero measure for $\mu_\bif$, even though they form a homogeneous dense subset of $\supp(\mu_\bif)$ of Hausdorff dimension $2(d-1)$.
\end{remark}

\bibliographystyle{short}
\bibliography{biblio}

\begin{center}
\rule{3cm}{0.5pt}
\end{center}

\textsc{\footnotesize LAMFA UMR-CNRS 7352, Universit\' e de Picardie Jules Verne, 33 rue Saint-Leu, 80039 Amiens Cedex 1, France.}\\
{\footnotesize {\em Email address:} \texttt{thomas.gauthier@u-picardie.fr}}

\end{document}